\documentclass[a4paper,11pt,oneside,headsepline]{scrartcl}

\usepackage{ifdraft}

\usepackage[utf8]{inputenc}
\usepackage[T1]{fontenc}

\usepackage{lmodern}
\usepackage{fourier}
\usepackage{amssymb, amsfonts}

\usepackage{mathrsfs} %
\usepackage{dsfont}
\usepackage[usenames,dvipsnames]{xcolor}%
\usepackage{lettrine}
\usepackage{url}

\usepackage{stmaryrd}

\usepackage[english]{babel}

\usepackage[draft=false]{scrlayer-scrpage}
\usepackage{wrapfig}
\usepackage{footmisc}
\usepackage{needspace}

\usepackage{amsmath}
\usepackage[thmmarks, amsmath, amsthm]{ntheorem}

\usepackage[style=numeric-comp,giveninits,url=false,sortcites=true,maxbibnames=99]{biblatex}
\bibliography{bibliography.bib}

\ifdraft{{}}
  {\usepackage[pdftex,
               pdfborder={0 0 0}, colorlinks=true,
               linkcolor=BrickRed, citecolor=ForestGreen, urlcolor=RoyalBlue]{hyperref}}

\ifdraft{
\usepackage[colorinlistoftodos]{todonotes}}{}

\ifdraft{%
\usepackage{lineno}}{}

\usepackage{booktabs}
\usepackage[inline]{enumitem}
\usepackage{float}
\usepackage{graphicx}
\usepackage{multicol}

\ifdraft{\date{\today}}{\date{}}

\ifdraft{%
\linenumbers
\newcommand*\patchAmsMathEnvironmentForLineno[1]{%
  \expandafter\let\csname old#1\expandafter\endcsname\csname #1\endcsname
  \expandafter\let\csname oldend#1\expandafter\endcsname\csname end#1\endcsname
  \renewenvironment{#1}%
     {\linenomath\csname old#1\endcsname}%
     {\csname oldend#1\endcsname\endlinenomath}}%
\newcommand*\patchBothAmsMathEnvironmentsForLineno[1]{%
  \patchAmsMathEnvironmentForLineno{#1}%
  \patchAmsMathEnvironmentForLineno{#1*}}%
\AtBeginDocument{%
\patchBothAmsMathEnvironmentsForLineno{equation}%
\patchBothAmsMathEnvironmentsForLineno{align}%
\patchBothAmsMathEnvironmentsForLineno{flalign}%
\patchBothAmsMathEnvironmentsForLineno{alignat}%
\patchBothAmsMathEnvironmentsForLineno{gather}%
\patchBothAmsMathEnvironmentsForLineno{multline}%
}}

\KOMAoptions{DIV=13}

\clearscrheadings
\automark[section]{subsection}

\renewcommand{\subsectionmark}[1]{}

\lohead{{\small \headertitle}}
\rohead{{\small \headerauthors}}

\RedeclareSectionCommand[%
  font=\Large\sffamily\bfseries,%
  beforeskip=1\baselineskip,%
  afterskip=0.5\baselineskip,%
  indent=0em%
  ]{section}

\RedeclareSectionCommands[%
  font=\normalfont\bfseries,%
  afterskip=-1em%
  ]{subsection,subsubsection}

\RedeclareSectionCommands[%
  font=\normalfont\itshape,%
  afterskip=-1em,%
  indent=0pt,
  ]{paragraph}

\AtEveryBibitem{\clearfield{isbn}}
\AtEveryBibitem{\clearlist{language}}
\AtEveryBibitem{\clearfield{pages}}

\newenvironment{enumeratearabic}{
\begin{enumerate}[label=(\arabic*), leftmargin=0pt,labelindent=2em,itemindent=!]
}{
\end{enumerate}
}

\newenvironment{enumeratearabic*}{
\begin{enumerate*}[label=(\arabic*)] %
}{
\end{enumerate*}
}

\newenvironment{enumerateroman*}{
\begin{enumerate*}[label=(\roman*)] %
}{
\end{enumerate*}
}

\numberwithin{equation}{section}

\newtheorem{theoremcounter}{theoremcounter}[section]

\theoremstyle{plain}

\newtheorem{lemma}[theoremcounter]{Lemma}

\newtheorem{theorem}[theoremcounter]{Theorem}

\theoremnumbering{Roman}
\newtheorem{maintheoremcounter}{maintheoremcounter}

\newtheorem{maintheorem}[maintheoremcounter]{Theorem}

\theoremstyle{definition}

\theoremstyle{remark}

\newtheorem{remark}[theoremcounter]{Remark}

\newtheorem*{mainremark}{Remark}

\newtheorem*{remarkcomputation}{Computation}

\newenvironment{mainremarkenumerate}
{
\begin{mainremark}
\begin{enumeratearabic}
}{
\end{enumeratearabic}
\end{mainremark}
}

\let\cal\undefined

\ifdraft{
 
}{
 
}

\newcommand{\tx}{\ensuremath{\text}}

\newcommand{\tbf}{\bfseries}

\newcommand{\cal}{\ensuremath{\mathcal}}
\renewcommand{\frak}{\ensuremath{\mathfrak}}

\newcommand{\cE}{\ensuremath{\cal{E}}}
\newcommand{\cF}{\ensuremath{\cal{F}}}

\newcommand{\frake}{\ensuremath{\frak{e}}}

\newcommand{\rmc}{\ensuremath{\mathrm{c}}}

\newcommand{\rmt}{\ensuremath{\mathrm{t}}}

\newcommand{\rmE}{\ensuremath{\mathrm{E}}}

\newcommand{\rmL}{\ensuremath{\mathrm{L}}}
\newcommand{\rmM}{\ensuremath{\mathrm{M}}}

\newcommand{\rmS}{\ensuremath{\mathrm{S}}}
\newcommand{\rmT}{\ensuremath{\mathrm{T}}}

\newcommand{\td}{\tilde}

\newcommand{\ov}{\overline}

\newcommand*{\longhookrightarrow}{\ensuremath{\lhook\joinrel\relbar\joinrel\rightarrow}}
\newcommand*{\longtwoheadrightarrow}{\ensuremath{\relbar\joinrel\twoheadrightarrow}}

\newcommand{\ra}{\ensuremath{\rightarrow}}
\newcommand{\hra}{\ensuremath{\hookrightarrow}}
\newcommand{\thra}{\ensuremath{\twoheadrightarrow}}

\newcommand{\lra}{\ensuremath{\longrightarrow}}
\newcommand{\lhra}{\ensuremath{\longhookrightarrow}}
\newcommand{\lthra}{\ensuremath{\longtwoheadrightarrow}}

\newcommand{\mto}{\ensuremath{\mapsto}}
\newcommand{\lmto}{\ensuremath{\longmapsto}}

\newcommand{\amid}{\ensuremath{\mathop{\mid}}}

\newcommand{\ZZ}{\ensuremath{\mathbb{Z}}}
\newcommand{\QQ}{\ensuremath{\mathbb{Q}}}
\newcommand{\RR}{\ensuremath{\mathbb{R}}}
\newcommand{\CC}{\ensuremath{\mathbb{C}}}

\renewcommand{\Re}{\ensuremath{\mathrm{Re}}}

\newcommand{\isdiv}{\amid}

\renewcommand{\pmod}[1]{\ensuremath{\;(\mathrm{mod}\, #1)}}

\newcommand{\Hom}{\ensuremath{\mathop{\mathrm{Hom}}}}

\newenvironment{psmatrix}{\left(\begin{smallmatrix}}{\end{smallmatrix}\right)}

\newcommand{\SL}[1]{\ensuremath{\mathrm{SL}_{#1}}}

\newcommand{\U}[1]{\ensuremath{\mathrm{U}_{#1}}}

\newcommand{\T}{\ensuremath{\rmt}}

\newcommand{\bbone}{\ensuremath{\mathds{1}}}

\renewcommand{\ker}{\ensuremath{\mathop{\mathrm{ker}}}}

\newcommand{\lspan}{\ensuremath{\mathop{\mathrm{span}}}}

\newcommand{\HS}{\mathbb{H}}

\newcommand{\lcm}{\ensuremath{\mathrm{lcm}}}
\newcommand{\parity}{\ensuremath{\mathrm{par}}}
\newcommand{\Res}{\ensuremath{\mathrm{Res}}}
\newcommand{\Ind}{\ensuremath{\mathrm{Ind}}}
\newcommand{\new}{\ensuremath{\mathrm{new}}}
\newcommand{\old}{\ensuremath{\mathrm{old}}}

\newcommand{\Ga}{\ensuremath{\Gamma}}
\newcommand{\ga}{\ensuremath{\gamma}}

\newcommand{\headertitle}{{\normalfont%
  All modular forms of weight $2$ can be expressed by Eisenstein series%
}}
\newcommand{\headerauthors}{%
  M.~Raum, J.~Xia%
}
\title{%
  All modular forms of weight~$2$ can be expressed by Eisenstein series%
}
\author{Martin Raum\thanks{The first author was partially supported by Vetenskapsr\aa det Grant~2015-04139.}%
\and%
Jiacheng Xia%
}

\begin{document}

\thispagestyle{scrplain}
\begingroup
\deffootnote[1em]{1.5em}{1em}{\thefootnotemark}
\maketitle
\endgroup

{\small
\noindent
{\tbf Abstract:}
We show that every elliptic modular form of integral weight greater than~$1$ can be expressed as linear combinations of products of at most two cusp expansions of Eisenstein series. This removes the obstruction of nonvanishing central $\rmL$-values present in all previous work. For weights greater than~$2$, we refine our result further, showing that linear combinations of products of exactly two cusp expansions of Eisenstein series suffice.
\\[.35em]
\textsf{\textbf{%
  central values of $\rmL$-functions%
}}%
\hspace{0.3em}{\tiny$\blacksquare$}\hspace{0.3em}%
\textsf{\textbf{%
  vector-valued Hecke operators%
}}%
\hspace{0.3em}{\tiny$\blacksquare$}\hspace{0.3em}%
\textsf{\textbf{%
  products of Eisenstein series%
}}
\\[0.15em]
\noindent
\textsf{\textbf{%
  MSC Primary:
  11F11%
}}
\hspace{0.3em}{\tiny$\blacksquare$}\hspace{0.3em}%
\textsf{\textbf{%
  MSC Secondary:
  11F67, 11F25
}}
}

\vspace{1.5em}

\Needspace*{4em}
\addcontentsline{toc}{section}{Introduction}
\markright{Introduction}
\lettrine[lines=2,nindent=.2em]{\tbf K}{ohnen-Zagier} proved in their work on periods of modular forms~\cite{kohnen-zagier-1984} that every modular form of level~$1$ can be written as a linear combination of products of at most two Eisenstein series. Their insight provides a precise connection between the resulting expressions for cuspidal Hecke eigenforms and the special values of the associated $\rmL$-functions. This connection stimulated subsequent work by, for instance, Borisov-Gunnells~\cite{borisov-gunnells-2001,borisov-gunnells-2001b,borisov-gunnells-2003}, Kohnen-Martin~\cite{kohnen-martin-2008}, the first named author~\cite{raum-2017}, and Dickson-Neururer~\cite{dickson-neururer-2018}, who investigated the case of higher levels. The nonvanishing of specific $\rmL$-values was crucial in all cases. For levels that are square-free away from at most two primes, Dickson-Neururer obtain a characterization of weight~$2$ newforms that can be expressed as a linear combination of products of at most two Eisenstein series for the congruence subgroup~$\Gamma_1(N)$. These are exactly those newforms whose central~$\rmL$-values do not vanish. In particular, results for newforms of weight~$2$ whose central~$\rmL$-values vanish are not included in any of the cited papers.

The condition on the central~$\rmL$-value for weight~$2$ newforms is a severe restriction in light of the Birch-Swinnerton-Dyer Conjecture which relates it to the rank of the Mordell-Weil groups of elliptic curves.  For instance, if a newform~$f$ of weight~$2$ has rational Fourier coefficients and negative Atkin-Lehner eigenvalue, it corresponds to an elliptic curve over~$\QQ$ with Mordell-Weil--rank at least~$1$ by work of Gross-Zagier~\cite{gross-zagier-1986}; See~\cite{goldfeld-1979,bhargava-shankar-2015,bhargava-skinner-2014} for a discussion of and results on distributions of ranks of elliptic curves. However, the case of vanishing central $\rmL$-values of weight~$2$ newforms is excluded from all available statements on products of Eisenstein series. In the present paper we close this gap; See Equation~\eqref{eq:thm:product-of-eisenstein-series:with-eisenstein} in Theorem~\ref{thm:product-of-eisenstein-series} and compare with the previously available assertion in Equation~\eqref{eq:dickson-neururer}.

Given positive integers~$k$ and~$N$, we denote by~$\cE_k(N)_\infty$ the space of functions spanned by Fourier expansions at~$\infty$ of all Eisenstein series of weight~$k$ and level~$N$, i.e., for $\Ga_1(N)$. The space of Fourier expansions at any cusp of all Eisenstein series of weight~$k$ and level~$N$ is denoted by~$\cE_k(N)$. As opposed to $\cE_k(N)_\infty$ it contains Fourier expansions that feature fractional exponents. Write $\rmM_k(\Ga)$ for the space of weight~$k$ modular forms for a group~$\Ga \subseteq \SL{2}(\ZZ)$ and $\rmM^\new_k(\Ga)$ for the subspace of newforms. The results of Dickson-Neururer, which hold if~$N$ is the product of two prime powers and a square-free integer, can be formulated as follows:
\begin{gather}
\label{eq:dickson-neururer}
\begin{alignedat}{3}
&
  \rmM_k(\Ga_0(N))
&&\;\subseteq\;
  \cE_{k}(N)_\infty
  \,+\,
  \sum_{l = 1}^{k-1}
  \cE_{k-l}(N)_\infty \,\cdot\, \cE_l(N)_\infty
\tx{,}\quad
&&
  \tx{if $k > 2$;}
\\
&
  \lspan \CC\big\{\, f \in \rmM^\new_2(\Ga_0(N)) \,:\, \rmL(f,1) \ne 0 \,\big\}
&&\;\subseteq\;
  \cE_2(N)_\infty
  \,+\,
  \cE_1(N)_\infty \,\cdot\, \cE_1(N)_\infty
\tx{.}
\end{alignedat}
\end{gather}

The main theorem of the present paper improves significantly on the second statement and drops completely the condition on~$N$. It also provides a variant of the first statement by suppressing the sum over weights~$l$, again without any condition on~$N$. One novel aspect of our main theorem is that we can omit Eisenstein series~$\cE_{k+l}(N)$ from the right hand side of Equation~\eqref{eq:thm:product-of-eisenstein-series:without-eisenstein}, which holds for~$k + l \ge 3$. Another one is that both~$k$ and~$l$ are fixed in Theorem~\ref{thm:product-of-eisenstein-series}.
\begin{maintheorem}
\label{thm:product-of-eisenstein-series}
Let~$k$, $l$, and~$N$ be positive integers. Then there is a positive integer~$N_0$ such that
\begin{gather}
\label{eq:thm:product-of-eisenstein-series:with-eisenstein}
  \rmM_{k+l}(\Ga(N))
\;\subseteq\;
  \cE_{k+l}(N)
  \,+\,
  \cE_k(N_0) \,\cdot\, \cE_l(N_0)
\tx{.}
\end{gather}
Moreover, if $k + l \ge 3$, then a suitable~$N_0$ is explicitly specified in Theorem~\ref{thm:cusp-forms-in-eisenstein-product} on page~\pageref{thm:cusp-forms-in-eisenstein-product}, and there is a positive integer~$N_1$---specified explicitly in Theorem~\ref{thm:eisenstein-series-in-eisenstein-product} on page~\pageref{thm:eisenstein-series-in-eisenstein-product}---such that
\begin{gather}
\label{eq:thm:product-of-eisenstein-series:without-eisenstein}
  \rmM_{k+l}(\Ga(N))
\;\subseteq\;
  \cE_k(\lcm(N_0,N N_1)) \,\cdot\, \cE_l(\lcm(N_0,N_1))
\tx{.}
\end{gather}
\end{maintheorem}
\begin{mainremarkenumerate}
\item 
Theorem~\ref{thm:product-of-eisenstein-series} is a consequence of Theorem~\ref{thm:cusp-forms-in-eisenstein-product} and Theorem~\ref{thm:eisenstein-series-in-eisenstein-product} in conjunction with Section~\ref{ssec:vector-valued-to-classical}, which revisits the connection between vector-valued and classical modular forms.

\item
Besides the case of principal congruence subgroups, Theorem~\ref{thm:cusp-forms-in-eisenstein-product} and Theorem~\ref{thm:eisenstein-series-in-eisenstein-product} also cover the cases of modular forms for~$\Ga_1(N)$, for~$\Ga_0(N)$, and for Dirichlet characters~$\chi$, and, most generally, of vector-valued modular forms. 

\item Theorem~\ref{thm:cusp-forms-in-eisenstein-product} and Theorem~\ref{thm:eisenstein-series-in-eisenstein-product} contain precise statements about which subspaces of the right hand sides of~\eqref{eq:thm:product-of-eisenstein-series:with-eisenstein} and~\eqref{eq:thm:product-of-eisenstein-series:without-eisenstein} equal which modular forms. For example, the space~$\rmM_k(\chi)$ of modular forms for a Dirichlet character~$\chi$ modulo~$N$ equals the following space of~$\Ga_0(N)$-invariants:
\begin{gather*}
  \rmM_k(\chi)
\;=\;
  \Big(
  \big(
  \cE_{k+l}(N)
  \,+\,
  \cE_k(N_0) \,\cdot\, \cE_l(N_0)
  \big) \otimes \chi
  \Big)^{\Ga_0(N)}
\tx{,}
\end{gather*}
where $\chi$ stands for the $\Ga_0(N)$ right representation $\begin{psmatrix} a & b \\ c & d \end{psmatrix} \mto \ov{\chi}(d)$ and $\Ga_0(N)$ acts on the spaces~$\cE_{k+l}(N)$, $\cE_k(N_0)$, and~$\cE_l(N_0)$ from the right via the usual slash actions~$|_{k+l}$, $|_k$, and $|_l$.

\item An explicit bound for~$N_0$ in the case of $k + l = 2$ could be obtained from an effective bound on gaps in the Fourier expansion of weight~$\frac{3}{2}$ modular forms.

\item Computer experiments for small~$N$ suggest that the second part of Theorem~\ref{thm:product-of-eisenstein-series} also holds true if~$k + l = 2$.
\end{mainremarkenumerate}

The first named author suggested in~\cite{raum-2017} that expressions for modular forms in terms of Eisenstein series can be employed to compute cusp expansions of modular forms of levels that are not square-free. Observe that algorithms rooted in modular symbols, which currently are the primary methods to compute elliptic modular forms, only reveal Fourier expansions at cusps mapped to~$\infty$ by Atkin-Lehner involutions. If the level is not square-free this is a proper subset of cusps. Cohen has implemented this idea in Pari/GP~\cite{belabas-cohen-2018,cohen-2019}. He built up on the results of Borisov-Gunnells, who restricted themselves to weights greater than~$2$. Theorem~\ref{thm:product-of-eisenstein-series} in this paper allows us to perform a similar computation of Fourier expansions of weight~$2$ modular forms. More precisely, since the action of~$\SL{2}(\ZZ)$ on the right hand side of~\eqref{eq:thm:product-of-eisenstein-series:with-eisenstein} and~\eqref{eq:thm:product-of-eisenstein-series:without-eisenstein} is known explicitly, Theorem~\ref{thm:product-of-eisenstein-series} yields a possibility to determine Fourier expansions of modular forms at all cusps. Tobias Magnusson and the first named author are preparing an implementation of this.

We now explain the three key differences of the present paper compared to previous work~\cite{kohnen-zagier-1984,borisov-gunnells-2001,borisov-gunnells-2001b,borisov-gunnells-2003,kohnen-martin-2008,dickson-neururer-2018}. The first key difference is the appearance of $\cE_k(N)$ as opposed to $\cE_k(N)_\infty$. A less general version of this was already used in~\cite{raum-2017}. Indeed, the vector-valued Hecke operator $\rmT_M$ in~\cite{raum-2017} produces from the Eisenstein series $E_k$ of level~$1$ the expansions at all cusps of the associated oldform~$E_k(M \,\cdot\,)$. The space spanned by the cusp expansion of a modular form at infinity, in general, does not carry an action of~$\SL{2}(\ZZ)$, but the space spanned by cusp expansions at all cusps does. Passing from $\cE_k(N)_\infty$ to $\cE_k(N)$ allows us to employ representation theoretic machinery and the theory of vector-valued Hecke operators developed in~\cite{raum-2017}.

The second key difference is that both~$k$ and~$l$ are fixed in Theorem~\ref{thm:product-of-eisenstein-series}, while $l$ must run in~\eqref{eq:dickson-neururer}. This directly impacts the strategy of proof, since the varying weight~$l$ gives access to almost the complete period polynomial as opposed to a single special $\rmL$-value. In the case of weight~$2$ modular forms, however, the approach of~\cite{rankin-1952,kohnen-zagier-1984} merely reveals parts of the period polynomial, excluding the central $\rmL$-value. Given a modular form, the vanishing of its central $\rmL$-value is not strong enough to imply the vanishing of the modular form. Among the innovations of~\cite{raum-2017} was to fix the weights of Eisenstein series, but vary their levels. This yields a relation to the nonvanishing problem for families of special $\rmL$-values, which can also be solved for weight~$2$ modular forms.

The third key difference is that the right hand side of~\eqref{eq:thm:product-of-eisenstein-series:without-eisenstein} displays only products of two Eisenstein series, omitting the additional space of weight~$(k+l)$ Eisenstein series. This yields a statement about the constant terms of products of Eisenstein series. If $k$ and $l$ are greater than~$2$, it can be derived without difficulties by multiplying suitable Eisenstein series of level~$N$. The cases of~$k \le 2$ or~$l \le 2$, however, require a more detailed analysis. We are leaving one open end in this context. The case of~$k = l = 1$ hinges on a precise understanding of tensor products of certain Weil representations, which we were not able to obtain here.

The proof of~\eqref{eq:thm:product-of-eisenstein-series:with-eisenstein} in Theorem~\ref{thm:cusp-forms-in-eisenstein-product} extends ideas in~\cite{raum-2017}. In particular, we have refined the argument at some places in order to obtain the explicit bound~$N_0$ for the level of Eisenstein series that appears on the right hand side  of~\eqref{eq:thm:product-of-eisenstein-series:with-eisenstein}. Our approach is based on a combination of the theory of vector-valued Hecke operators~\cite{raum-2017} and the Rankin-Selberg method~\cite{rankin-1952,kohnen-zagier-1984}.

The proof of~\eqref{eq:thm:product-of-eisenstein-series:without-eisenstein} in Theorem~\ref{thm:eisenstein-series-in-eisenstein-product} builds up on the statement of Theorem~\ref{thm:cusp-forms-in-eisenstein-product}. The methods that we employ are quite different, however. Specifically, we examine the spaces~$\cE_{k}(N)$, $\cE_{l}(N)$, and $\cE_{k+l}(N)$ as $\SL{2}(\ZZ)$-representations. Their subspaces of vectors fixed by the action of~$T = \begin{psmatrix} 1 & 1 \\ 0 & 1 \end{psmatrix}$ are related to the spaces spanned by constant terms of Eisenstein series. Then Theorem~\ref{thm:eisenstein-series-in-eisenstein-product} follows from an argument from representation theory.

\section{Preliminaries}

We write~$\HS$ for the Poincar\'e upper half plane, which carries an action of~$\SL{2}(\RR)$ by M\"obius transformations. We fix the notation $T = \begin{psmatrix} 1 & 1 \\ 0 & 1 \end{psmatrix} \in \SL{2}(\RR)$ for the transformation acting on~$\HS$ as a translation by~$1$. We write $\Ga_\infty^+ \subset \SL{2}(\ZZ)$ for the subgroup generated by~$T$.

\subsection{Arithmetic types}

An arithmetic type is a finite dimensional, complex representation of~$\SL{2}(\ZZ)$. The representation space of an arithmetic type~$\rho$ is denoted by~$V(\rho)$. We call an arithmetic type a congruence type if its kernel is a congruence subgroup. The level of a congruence type is the level of its kernel. We record that all congruence types are unitarizable.

The trivial, one-dimensional, complex representation of a subgroup~$\Ga \subseteq \SL{2}(\ZZ)$ will be denoted by~$\bbone_\Ga$. Usually, $\Ga$ is clear from the context and we abbreviate~$\bbone_\Ga$ by~$\bbone$.

The induction of arithmetic types is explained in detail in~\cite{raum-2017} using a choice of representatives. We set
\begin{gather*}
  \rho^\times_N
\;:=\;
  \Ind_{\Ga_1(N)}^{\SL{2}(\ZZ)}\, \bbone
\tx{.}
\end{gather*}
Recall from, for example, \cite{carnahan-2012} that
\begin{gather}
\label{eq:ind_Ga1_decomposition}
  \rho^\times_N
\;\cong\;
  \bigoplus_\chi \Ind_{\Ga_0(N)}^{\SL{2}(\ZZ)}\,\chi
\tx{,}
\end{gather}
where $\chi$ runs through Dirichlet characters mod~$N$ considered as representations of~$\Ga_0(N)$ via the assignment $\begin{psmatrix} a & b \\ c & d \end{psmatrix} \mto \chi(d)$. We write~$\rho_\chi$ for its induction to~$\SL{2}(\ZZ)$.

Observe that by Frobenius reciprocity, we have
\begin{gather*}
  \Hom_{\SL{2}(\ZZ)} \big( \bbone, \rho^\times_N \big)
\;\cong\;
  \Hom_{\Ga_{1}(N)} ( \bbone, \bbone )
\tx{,}
\end{gather*}
which is one-dimensional. In particular, $\rho^\times_N$ contains a unique copy of the trivial representation. We write
\begin{gather*}
  \rho^\times_N \ominus \bbone
\end{gather*}
for its orthogonal complement. For this purpose, we choose any unitary structure of the representation~$\rho$, since the result is independent of it.

We say that an arithmetic type~$\rho$ has~$T$-fixed vectors, if there is a nonzero vector~$v \in V(\rho)$ such that $\rho(T) v = v$. The subrepresentation of~$\rho$ on which~$\begin{psmatrix} -1 & 0 \\ 0 & -1 \end{psmatrix}$ in the center of~$\SL{2}(\ZZ)$ acts by~$\pm 1$ is denoted by~$\rho^\pm$. Write $\rho^T$ for the space of~$\Ga_\infty^+$-invariants in~$V(\rho)$. The intersection of~$\rho^T$ with~$\rho^\pm$ is denoted by~$\rho^{T\,\pm}$. We let $\parity(k) = \pm$ be the parity of~$k$, so that $\rho^{\parity(k)}$ denotes the subrepresentation of~$\rho$ on which $\begin{psmatrix} -1 & 0 \\ 0 & -1 \end{psmatrix}$ acts by~$(-1)^k$.

\subsection{Modular forms}

The classical slash actions for~$k \in \ZZ$,
\begin{gather*}
  \big( f \big|_k\, \begin{psmatrix} a & b \\ c & d \end{psmatrix} \big)(\tau)
:=
  (c \tau + d)^{-k}\, f\big(\begin{psmatrix} a & b \\ c & d \end{psmatrix} \tau \big)
\tx{,}
\end{gather*}
extend to vector-valued slash actions
\begin{gather*}
  \big( f \big|_{k,\rho}\, \begin{psmatrix} a & b \\ c & d \end{psmatrix} \big)(\tau)
:=
  (c \tau + d)^{-k}\,
  \rho\big( \begin{psmatrix} a & b \\ c & d \end{psmatrix}^{-1} \big)
  f\big(\begin{psmatrix} a & b \\ c & d \end{psmatrix} \tau \big)
\tx{.}
\end{gather*}

The space of classical modular forms~$\rmM_k(\Ga)$ for a subgroup~$\Ga \subseteq \SL{2}(\ZZ)$ is the space of holomorphic functions~$f :\, \HS \ra \CC$ such that
\begin{enumerateroman*}
\item $f \big|_k\,\ga = f$ for all~$\ga \in \Ga$ and
\item $(f \big|_k\,\ga)(\tau)$ is bounded as $\tau \ra i \infty$ for all~$\ga \in \SL{2}(\ZZ)$.
\end{enumerateroman*}
The subspace~$\rmS_k(\Ga)$ of cusp forms is defined as the space of modular forms that satisfy the stronger second condition $(f \big|_k\,\ga)(\tau) \ra 0$ as $\tau \ra i \infty$ for all~$\ga \in \SL{2}(\ZZ)$.

Recall that we set $\chi(\begin{psmatrix} a & b \\ c & d \end{psmatrix}) = \chi(d)$ for~$\begin{psmatrix} a & b \\ c & d \end{psmatrix} \in \Ga_0(N)$ and a Dirichlet character~$\chi$ modulo~$N$. The space~$\rmM_k(\chi)$ is defined as the space of holomorphic functions~$f :\, \HS \ra \CC$ such that
\begin{enumerateroman*}
\item $f \big|_k\,\ga = \chi(\ga) f$ for all~$\ga \in \Ga_0(N)$ and
\item $(f \big|_k\,\ga)(\tau)$ is bounded as $\tau \ra i \infty$ for all~$\ga \in \SL{2}(\ZZ)$.
\end{enumerateroman*}

Similarly, the spaces of vector-valued modular forms~$\rmM_k(\rho)$ and cusp forms~$\rmS_k(\rho)$ for an arithmetic type~$\rho$ are defined as the spaces of holomorphic functions~$f :\, \HS \ra V(\rho)$ such that
\begin{enumerateroman*}
\item $f \big|_{k,\rho}\,\ga = f$ for all~$\ga \in \SL{2}(\ZZ)$ and
\item $f(\tau)$ is bounded (with respect to some norm on~$V(\rho)$) as $\tau \ra i \infty$ if $f \in \rmM_k(\rho)$ or $f(\tau) \ra 0$ as $\tau \ra i \infty$ if $f \in \rmS_k(\rho)$.
\end{enumerateroman*}

If $\Ga$ is a congruence subgroup, we write $\rmS^\old_k(\Ga) \subseteq \rmS_k(\Ga)$ for the space of oldforms and $\rmS^\new_k(\Ga) \subseteq \rmS_k(\Ga)$ for the set of (normalized) newforms.

There are natural bases for~$\rho^\times_N$, $\rho_\chi$, and their dual representations that are indexed by (a choice of representatives for) the cosets~$\Ga_1(N) \backslash \SL{2}(\ZZ)$ and~$\Ga_0(N) \backslash \SL{2}(\ZZ)$. Specifically, the representation spaces of~$\rho^\times_N$ and $\rho_\chi$ are
\begin{gather}
  V(\rho^\times_N)
\;=\;
  \lspan\CC\,\big\{ \frake_\ga \,:\, \ga \in \Ga_1(N) \backslash \SL{2}(\ZZ) \big\}
\quad\tx{and}\quad
  V(\rho^\times_N)
\;=\;
  \lspan\CC\,\big\{ \frake_\ga \,:\, \ga \in \Ga_0(N) \backslash \SL{2}(\ZZ) \big\}
\tx{.}
\end{gather}
Given a modular form for any of these types, we refer to the component associated with the trivial coset as the component at infinity.

The vector-valued Hecke operators~$\rmT_M$ including their basic properties were introduced in~\cite{raum-2017}. Given a representation~$\rho$, they yield a representation on the vector space
\begin{gather}
  \lspan\CC\,\big\{ \frake_m \,:\, \ga \in \Delta_M \big\}
  \,\otimes\,
  V(\rho)
\tx{,}\quad
  \Delta_M
:=
  \big\{
  m = \begin{psmatrix} a & b \\ 0 & d \end{psmatrix} \,:\,
  a,b,d \in \ZZ,\, 0 < a, d,\, 0 \le b < d,\, ad = M
  \big\}
\tx{.}
\end{gather}
We will give more precise references to the properties of vector-valued Hecke operators when we employ them. At some point in the proof of Theorem~\ref{thm:cusp-forms-in-eisenstein-product}, we will use~$\rmT_M$ to denote the classical Hecke operator, but otherwise it is the vector-valued one.

As a Sturm bound for modular forms in~$\rmM_k(\chi)$, where $\chi$ is a Dirichlet character modulo~$N$, we use
\begin{gather}
\label{eq:def:sturm-bound}
  B(k, N)
\;:=\;
  \Big\lceil
  \frac{k}{12}\,
  N \prod_{\substack{p \isdiv N \\ \tx{$p$ prime}}}\big( 1 + \tfrac{1}{p} \big)
  \Big\rceil
\tx{.}
\end{gather}

\section{Congruence types and their modular forms}

We start this section with a characterization of congruence types that are generated by their~$T$-fixed vectors.
\begin{lemma}
\label{la:T-fixed-vectors-yield-embedding}
Let $\rho$ be an irreducible congruence type of level~$N$ and assume that $\rho$ has $T$-fixed vectors. Then there is an embedding
\begin{gather*}
  \rho
\lhra
  \Ind_{\Ga_1(N)}\,\bbone
\tx{.}
\end{gather*}
\end{lemma}
\begin{proof}
This is a straightforward application of Frobenius reciprocity. Observe that~$T \in \SL{2}(\ZZ)$ generates~$\Ga_\infty^+$ and that $\Ga_1(N)$ is generated by~$\Ga_\infty^+$ and $\Ga(N) \subseteq \ker(\rho)$. The assumption that $\rho$ be generated by its $T$-fixed vectors can thus be rephrased as
\begin{gather*}
  0
\;\ne\;
  \Hom_{\Ga_\infty^+}\big( \bbone,\, \Res_{\Ga_\infty^+}\,\rho \big)
\;\cong\;
  \Hom_{\Ga_1(N)}\big( \bbone,\, \Res_{\Ga_1(N)}\,\rho \big)
\;\cong\;
  \Hom_{\SL{2}(\ZZ)}\big( \Ind_{\Ga_1(N)} \bbone,\, \rho \big)
\tx{.}
\end{gather*}
Since~$\Ind_{\Ga_1(N)} \bbone$ is unitarizable, this implies that~$\rho$ occurs among its direct summands.
\end{proof}

The next lemma later allows us to focus on congruence types that are generated by their $T$-fixed vectors, so that we can invoke Lemma~\ref{la:T-fixed-vectors-yield-embedding}.
\begin{lemma}
\label{la:hecke-operators-forcing-T-fixed-vectors}
Let~$\rho$ be a congruence type of level~$N$. Then, for every positive integer~$M$ that is divisible by~$N$, there is a subrepresentation~$\rho' \subseteq \rmT_M \rho$ that is generated by its~$T$-fixed vectors and that satisfies $\rho \subseteq \rmT_M\, \rho'$.
\end{lemma}
\begin{proof}
Since vector-valued Hecke operators intertwine with direct sums of arithmetic types, we can and will assume that~$\rho$ is irreducible. By the assumptions $N \isdiv M$, and therefore for any~$v \in V(\rho)$, we have the invariance
\begin{gather*}
  \big( \rmT_M\,\rho \big)(T) \Big( v \otimes \frake_{\begin{psmatrix} M & 0 \\ 0 & 1 \end{psmatrix}} \Big)
=
  \big( \rho(T^M)  v \big) \otimes \frake_{\begin{psmatrix} M & 0 \\ 0 & 1 \end{psmatrix}}
=
  v \otimes \frake_{\begin{psmatrix} M & 0 \\ 0 & 1 \end{psmatrix}}
\tx{.}
\end{gather*}
In particular, $\rmT_M\,\rho$ contains nonzero $T$-fixed vectors.

We record that there is an injection~$\rho \hra \rmT_M (\rmT_M\, \rho)$ by Proposition~2.10 of~\cite{raum-2017}. Consider an arbitrary irreducible subrepresentation~$\rho'$ of~$\rmT_M\, \rho$. Since $\rmT_M\, \rho$ is unitarizable by Lemma~2.4 of~\cite{raum-2017}, we conclude that $\rho'$ is also a quotient of $\rmT_M\, \rho$. Specifically, there is a surjective homomorphism~$\phi :\, \rmT_M\, \rho \thra \rho'$. We will need the concrete shape of~$\phi$. Given~$m \in \Delta_M$, where $\Delta_M$ is as in Section~2.1 of~\cite{raum-2017}, there is a linear map~$\phi_m : V(\rho) \ra V(\rho')$ such that $\phi(v \otimes \frake_m) = \phi_m(v)$ for every~$v \in V(\rho)$. Since $\phi$ is nonzero, there is at least one~$m$ such that $\phi_m$ is not zero.

We claim that $\rho \subseteq \rmT_M\, \rho'$. For a proof, we consider the following composition of homomorphisms, where the first and second one arise from Proposition~2.10 and Proposition~2.5 of~\cite{raum-2017}:
\begin{gather*}
  \rho
\lhra
  \rmT_M \big( \rmT_M\, \rho \big)
\lthra
  \rmT_M\, \rho'
\tx{.}
\end{gather*}
We have to demonstrate that this composition is not zero. To this end, we let $v \in V(\rho)$ and inspect to what element of~$V(\rmT_M\,\rho')$ it is mapped. We write $m^\#$ for the adjugate of~$m \in \Delta_M$ and obtain
\begin{gather*}
  v
\lmto
  \sum_{m \in \Delta_M}
  v \otimes \frake_m \otimes \frake_{m^\#}
\lmto
  \sum_{m \in \Delta_M}
  \phi_m(v) \otimes \frake_{m^\#}
\tx{,}
\end{gather*}
which does not vanish for all~$v$, since~$\phi_m$ is nonzero for some~$m$.

To finish the proof, it suffices to choose some irreducible~$\rho' \subseteq \rmT_M\,\rho$ that contains a nonzero $T$-fixed vector. Since~$\rho$ is irreducible, it is generated by this vector.
\end{proof}

\subsection{Passing from vector-valued to classical modular forms}
\label{ssec:vector-valued-to-classical}

There is a connection between vector-valued modular forms and classical modular forms via induction. For instance, $\rmM_k(\chi)$ for a Dirichlet character~$\chi$ to~$\rmM_k(\rho_\chi)$ are related to each other by Proposition~1.5 of~\cite{raum-2017}. The present section likewise connects Theorem~\ref{thm:product-of-eisenstein-series} to Theorems~\ref{thm:cusp-forms-in-eisenstein-product} and~\ref{thm:eisenstein-series-in-eisenstein-product}.  

Throughout, we fix an integer~$k$. Given a Dirichlet character~$\chi$, recall the map
\begin{gather}
  \Ind :\,
  \rmM_k(\chi)
\lra
  \rmM_k(\rho_\chi)
\tx{,}\quad
  f
\lmto
  \sum_{\ga \in \Ga_0(N) \slash \SL{2}(\ZZ)}
  (f \slash_k\,\ga) \frake_\ga
\end{gather}
in~(1.3) of~\cite{raum-2017} and Proposition~1.5 of~\cite{raum-2017}, asserting that $\Ind$ is an isomorphism. The argument extends to the map
\begin{gather}
\label{eq:induction-of-modular-forms}
  \Ind :\,
  \rmM_k(\Ga)
\lra
  \rmM_k\big( \Ind_\Ga^{\SL{2}(\ZZ)}\,\bbone \big)
\tx{,}\quad
  f
\lmto
  \sum_{\ga \in \Ga \slash \SL{2}(\ZZ)}
  (f \slash_k\,\ga) \frake_\ga
\end{gather}
for any finite index subgroup~$\Ga \subseteq \SL{2}(\ZZ)$.

Let~$\cF$ be a finite dimensional space of functions on~$\HS$ on which~$\SL{2}(\ZZ)$ acts by the slash action of weight~$k$. Then the induction map yields isomorphisms
\begin{gather}
\label{eq:component-at-infty}
  \big( \cF \otimes \chi \big)^{\Ga_0(N)}
\lra
  \big( \cF \otimes \rho_\chi \big)^{\SL{2}(\ZZ)}
\quad\tx{and}\quad
  \cF^\Ga
\lra
  \big( \cF \otimes \Ind_{\Ga}^{\SL{2}(\ZZ)}\,\bbone \big)^{\SL{2}(\ZZ)}
\tx{,}
\end{gather}
where~$\chi$ and~$\Ga$ are as before. Their inverse maps are the projections to the component at infinity. The proof is analogous to the one of Proposition~1.5 of~\cite{raum-2017}.

\subsection{Components of Eisenstein series}

Given an integer~$k$ and an arithmetic type~$\rho$, we write
\begin{gather*}
  \cE_k(\rho)
\;:=\;
  \big\{
  v \circ f \,:\,
  f \in \rmE_k(\rho),\, v \in V(\rho)^\vee
  \big\}
\end{gather*}
for the components of Eisenstein series of weight~$k$ and type~$\rho$. There is an action of $\SL{2}(\ZZ)$ on $\cE_k(\rho)$ via its action on~$V(\rho)^\vee$. We infer from~\eqref{eq:induction-of-modular-forms} that~$\cE_k(N) := \cE_k(\rho^\times_N)$.

The next lemma is a variation of Proposition~3.4 of~\cite{raum-2017}.
\begin{lemma}
\label{la:eisenstein-series-hecke-stable}
Let $N, M \ge 1$. Then for every~$k \ge 1$ we have
\begin{gather*}
  \rmT_M\, \cE_k(N)
\subseteq
  \cE_k(M N)
\tx{.}
\end{gather*}
\end{lemma}
\begin{proof}
Given any pair of integers $(c', d')$ and any nonnegative integers~$\alpha,\beta,\delta$ with $\alpha \delta = M$, we see that, for all $s \in \CC$ satisfying~$k + 2 \Re(s) > 2$,
\begin{align*}
&\hphantom{{}=}
  G_{k,N,c',d'}(\tau, s)
  \big|_k\, \begin{psmatrix} \alpha & \beta \\ 0 & \delta \end{psmatrix}
=
  \sum_{\substack{(c,d) \in \ZZ^2 \setminus (0,0) \\ c \equiv c' \pmod{N} \\ d \equiv d' \pmod{N}}}
  (c \tau + d)^{-k} |c \tau + d|^{-2s}
  \big|_k\, \begin{psmatrix} \alpha & \beta \\ 0 & \delta \end{psmatrix}
\\
&{}=
  \delta^{2s}
  \sum_{\substack{(c,d) \in \ZZ^2 \setminus (0,0)\\ c \equiv c' \pmod{N} \\ d \equiv d' \pmod{N}}}
  (\alpha c \tau + \delta d + \beta c)^{-k}
  |\alpha c \tau + \delta d + \beta c|^{-2s}
\\
&{}=
  \delta^{2s}
  \sum_{\substack{
     c'', d'' \pmod{M N} \\
     c'' \equiv \alpha c' \pmod{\alpha N} \\
     d'' \equiv \delta d' + \beta c' \pmod{\gcd(\beta,\delta)N}}}\;
  \sum_{\substack{c,d \in \ZZ^2 \setminus (0,0) \\ c \equiv c'' \pmod{M N} \\ d \equiv d'' \pmod{M N}}}
  (c \tau + d)^{-k} |c \tau + d|^{-2s}
\tx{.}
\end{align*}
If $k \ne 2$ the right hand side is analytic at~$s = 0$ and yields a linear combination of holomorphic Eisenstein series~$G_{k,MN,c'',d''}(\tau)$. We thus obtain the statement directly. In the case of~$k = 2$, a linear combination of the~$G_{2,N,c',d'}(\tau,0)$ lies in $\cE_2(N)$ if and only if its image under the $\xi_2$-operator of~\cite{bruinier-funke-2004} vanishes. Since the $\xi$-operator intertwines with the action of~$\SL{2}(\RR)$, this finishes the proof.
\end{proof}

\section{Cusp forms}

The first lemma of this section is a variant of, for instance, Propositions~3.7 and~4.1 in~\cite{raum-2017} or Corollary~4.2 in~\cite{dickson-neururer-2018}. It relates products of Eisenstein series to special values of~$\rmL$-functions, generalizing results by Rankin~\cite{rankin-1952} and various other authors. We omit the proof, which is mutatis mutandis the one of Corollary~4.2 of~\cite{dickson-neururer-2018}.
\begin{lemma}
\label{la:petersson-scalar-product-against-eisenstein-product}
Fix integers $k, l \ge 1$ and a Dirichlet character~$\chi$ mod~$N$ for some positive integer~$N$. Let~$\psi$ be a primitive Dirichlet character of modulus dividing~$N$ satisfying $\psi(-1) = (-1)^k$. If $k = 2$ assume that $\psi \ne \bbone$, and if $l \le 2$ assume that $\psi \ne \chi$. Then there is
\begin{gather*}
  g
\;\in\;
  \big( \cE_k(N)_\infty \otimes \cE_l(N)_\infty \otimes \chi \big)^{\Ga_0(N)}
\;\subseteq\;
  \rmM_{k+l}(\chi)
\end{gather*}
such that for every~$f \in \rmS_{k+l}(\chi)$, we have
\begin{gather}
\label{eq:la:petersson-scalar-product-against-eisenstein-product}
  \big\langle g,\, f \big\rangle
\;=\;
  c \cdot
  \rmL\big( f^\rmc, k+l-1 \big)\,
  \rmL\big( f^\rmc \otimes \psi, l \big)
\end{gather}
for a nonzero constant~$c$. Here $f^\rmc(\tau) := \ov{f(-\ov{\tau})}$ denotes the modular form whose Fourier coefficients are the complex conjugate of those of~$f$.
\end{lemma}

While Lemma~\ref{la:petersson-scalar-product-against-eisenstein-product} accounts solely for the analytic machinery required for the proof of Theorem~\ref{thm:cusp-forms-in-eisenstein-product}, we split the representation theoretic instrumentation into several separate statements. We will employ the next lemma to discard the space of oldforms from our considerations, so that we can focus on the span of newforms.
\begin{lemma}
\label{la:oldforms}
Let $M, N, N'$ be positive integers and~$f \in \rmM^\old_k(\Ga_1(N))$ an oldform. Assume that there is a modular form~$g$ of level~$N'$ such that $f(\tau) = g(M \tau)$. Then we have
\begin{gather*}
  \Ind(f)
\;\in\;
  \rmT_M\, \rmM_k\big( \rho^\times_{N'} \big)
\tx{.}
\end{gather*}
\end{lemma}
\begin{proof}
By Proposition~1.5 of~\cite{raum-2017}, we have $\Ind(g) \in \rmM_k(\rho^\times_{N'})$.
 Now the lemma is a special case of Proposition~2.17 of~\cite{raum-2017}.
\end{proof}

A key difference between~\cite{raum-2017} and~\cite{dickson-neururer-2018} is that the latter employs the products of Eisenstein series of varying weight, while the weight of Eisenstein series is fixed in~\cite{raum-2017} and the level of Eisenstein series varies. We follow~\cite{raum-2017} in this paper. The argument in~\cite{dickson-neururer-2018} builds up on the vanishing of (twisted) period polynomials, which is a strategy that goes back to~\cite{rankin-1952,kohnen-zagier-1984}. This reasoning was replaced in~\cite{raum-2017} by an inspection of twists of one special $\rmL$-value, whose simultaneous vanishing can be controlled by means of results by Waldspurger and Kohnen-Zagier~\cite{waldspurger-1981,kohnen-zagier-1984}. If~$k + l = 2$, we need a result of Ono and Skinner~\cite{ono-skinner-1998}, which builds up on these and work of Friedberg-Hoffstein~\cite{friedberg-hoffstein-1995}. The following lemma provides the required relation between vector-valued Hecke operators and twists of modular forms, which yield twisted $\rmL$-values.
\begin{lemma}
\label{la:twist-by-dirichlet-character}
Let $M, N$ be positive integers, $f \in \rmM_k(\Ga_1(N))$, and~$\chi$ a Dirichlet character mod~$M$. Then we have
\begin{gather*}
  \Ind( f \otimes \chi )
\;=\;
  \phi\big(
  \rmT_{M^2}\, \Ind(f)
  \big)
\end{gather*}
for a homomorphism $\phi :\, \rmT_{M^2}\, \rho^\times_N \ra \rho^\times_{M^2 N}$ of arithmetic types.
\end{lemma}
\begin{proof}
This is a special case of Proposition~2.19 of~\cite{raum-2017}. Specifically, the map~$\phi$ arises from~$\pi_{\mathrm{twist}}$ in that proposition after decomposing $\rho^\times_N$ and $\rho^\times_{M^2 N}$ as in~\eqref{eq:ind_Ga1_decomposition}.
\end{proof}

\begin{theorem}
\label{thm:cusp-forms-in-eisenstein-product}
Let~$k,l \ge 1$ and $\rho$ be a congruence type of level~$N$. Then we have
\begin{gather}
\label{eq:thm:cusp-forms-in-eisenstein-product}
  \rmS_{k+l}(\rho)
\;\subseteq\;
  \big( \cE_k(N_0) \otimes \cE_l(N_0) \otimes \rho \big)^{\SL{2}(\ZZ)}
\end{gather}
for~$N_0$ chosen as follows in terms of the Sturm bound in~\eqref{eq:def:sturm-bound}, if $\rho$ is generated by its $T$-fixed vectors:
\begin{alignat*}{2}
  N_0
&\;=\;
  N B(k+l, N)
\tx{,}\quad
&&
  \tx{if $k \ne l$, $k$ even;}
\\
  N_0
&\;=\;
  16 N B(k+l, 16 N))
\tx{,}\quad
&&
  \tx{if $k \ne l$, $k$ odd;}
\\
  N_0
&\;=\;
  N (16 B(k+\tfrac{1}{2}, N))^4 B\big(k+l, (16 B(k+\tfrac{1}{2}, N))^4 \big)
\tx{,}\quad
&&
  \tx{if $k = l \ge 2$.}
\end{alignat*}
For general $\rho$, the above bounds for~$N_0$ are multiplied by an additional factor~$N$:
\begin{alignat*}{2}
  N_0
&\;=\;
  N^2 B(k+l, N)
\tx{,}\quad
&&
  \tx{if $k \ne l$, $k$ even;}
\\
  N_0
&\;=\;
  16 N^2 B(k+l, 16 N))
\tx{,}\quad
&&
  \tx{if $k \ne l$, $k$ odd;}
\\
  N_0
&\;=\;
  N^2 (16 B(k+\tfrac{1}{2}, N))^4 B\big(k+l, (16 B(k+\tfrac{1}{2}, N))^4 \big)
\tx{,}\quad
&&
  \tx{if $k = l \ge 2$.}
\end{alignat*}
If~$k = l = 1$, there exists some~$N_0$ such that~\eqref{eq:thm:cusp-forms-in-eisenstein-product} holds.
\end{theorem}
\begin{remark}
The factor of~$B(k+\frac{1}{2},N)^4$ in the case of~$k = l$ arises (twice) because we restrict in the proof to twists of $\rmL$-values by Kronecker characters associated with imaginary quadratic fields. Explicit non-vanishing bounds for central values of the families $\rmL(f \otimes \psi, s)$, $f$ a fixed newform and $\psi$ a Dirichlet character can be used to improve this bound.
\end{remark}
\begin{proof}
Suppose that the statement holds for all~$\rho$ that are generated by their~$T$-fixed vectors. Then the general case follows by applying Lemma~\ref{la:hecke-operators-forcing-T-fixed-vectors} in conjunction with Lemma~\ref{la:eisenstein-series-hecke-stable} in this paper and Theorem~2.8 of~\cite{raum-2017}. In particular, we can and will assume that $\rho$ is generated by its~$T$-fixed vectors. Specifically, it suffices to treat all subrepresentations of~$\rho^\times_N = \Ind_{\Ga_1(N)}\,\bbone$ by Lemma~\ref{la:T-fixed-vectors-yield-embedding}. Recall that we have $\rho^\times_N = \oplus_\chi \rho_\chi$, where $\chi$ runs through all Dirichlet characters mod~$N$. For this reason, we can further restrict to the subrepresentations of~$\rho_\chi = \Ind_{\Ga_0(N)}\,\chi$, where $\chi$ is an arbitrary Dirichlet characters mod~$N$.

Using induction on~$N$, we can use Lemma~\ref{la:oldforms} to obtain the inclusion
\begin{gather}
\label{eq:prop:cusp-forms-in-eisenstein-product:oldforms}
  \Ind\big( \rmS^\old_{k+l}(\chi) \big)
\subseteq
  \big( \cE_k(N_0) \otimes \cE_l(N_0) \otimes \rho \big)^{\SL{2}(\ZZ)}
\tx{.}
\end{gather}
Hence we can and will restrict to newforms in the remainder of the proof. Consider the subset (as opposed to subspace)~$\rmS^\new_{k+l}(\chi) \subseteq \rmS_{k+l}(\chi)$ of newforms. Since the right hand side is a vector space, the statement of the theorem follows, if we show that
\begin{gather}
\label{eq:prop:cusp-forms-in-eisenstein-product:newforms}
  \Ind\big( \rmS^\new_{k+l}(\chi) \big)
\subseteq
  \big( \cE_k(N_0) \otimes \cE_l(N_0) \otimes \rho \big)^{\SL{2}(\ZZ)}
\tx{.}
\end{gather}

We next inspect the Petersson scalar products of newforms for the Dirichlet character~$\chi$ and the component at infinite~$g_\infty$ of elements~$g$ of~$(\cE_k(N') \otimes \cE_l(N') \otimes \rho)^{\SL{2}(\ZZ)}$ for some positive integer~$N'$. By the isomorphisms~\eqref{eq:component-at-infty}, this is equivalent to choosing a classical modular forms $g_\infty \in (\cE_k(N') \otimes \cE_l(N') \otimes \chi)^{\Ga_0(N)}$. Consider an arbitrary newform~$f \in \rmS^\new_{k+l}(\chi)$ and assume that~$\langle g_\infty, f \rangle \ne 0$ for some~$g_\infty \in (\cE_k(N') \otimes \cE_l(N') \otimes \chi)^{\Ga_0(N)}$. Decomposing~$g_\infty$ as a sum of Hecke eigenforms, we find that there is a linear combination~$\sum c(M) \rmT_M$, $M$ running through a suitable range of positive integers, of classical Hecke operators~$\rmT_M$ such that
\begin{gather*}
  f
\;=\;
  g_\infty \Big|_{k+l} \sum c(M) \rmT_M
\tx{.}
\end{gather*}
By virtue of the Sturm bound for modular forms of weight~$k+l$ and level~$N'$, we can further assume that $c(M) = 0$ if $M$ is larger than~$B(k+l,N')$ in~\eqref{eq:def:sturm-bound}

We summarize that, given a positive integer~$N'$ divisible by~$N$ such that there is~$g_\infty \in (\cE_k(N') \otimes \cE_l(N') \otimes \chi)^{\Ga_0(N)}$ with~$\langle g_\infty,\, f \rangle \ne 0$, we have
\begin{gather*}
  \Ind(f)
\in
  \sum_{M=1}^{B(k+l,N')} \rmT_M
  \big( \cE_k(N') \otimes \cE_l(N') \otimes \rho \big)^{\SL{2}(\ZZ)}
\tx{.}
\end{gather*}
Proposition~2.5 of~\cite{raum-2017} and Lemma~\ref{la:eisenstein-series-hecke-stable} imply the inclusion
\begin{gather*}
  \Ind(f)
\in
  \big( \cE_k(N' B(k+l,N')) \otimes \cE_l(N' B(k+l,N')) \otimes \rho \big)^{\SL{2}(\ZZ)}
\tx{.}
\end{gather*}
To finish the proof, we have to show the existence of~$g_\infty$ for $N'$ suitable to match the statement of Theorem~\ref{thm:cusp-forms-in-eisenstein-product}. By symmetry of the assertion in~$k$ and~$l$, we can and will assume that $l \ge k$.

We next utilize Lemma~\ref{la:petersson-scalar-product-against-eisenstein-product}. Observe that the complex conjugate~$f^\rmc$ that appears in Lemma~\ref{la:petersson-scalar-product-against-eisenstein-product} lies in~$\rmS_{k+l}(\ov\chi)$. The special $\rmL$-values $\rmL(f^\rmc, s)$ and $\rmL(f^\rmc \otimes \psi, s)$ do not vanish if $s > 1 + (k+l) \slash 2$, since they can be expressed in terms of a convergent Euler product. If $s = 1 + (k+l) \slash 2$, they lie on the abscissa of convergence of~$\rmL(f^\rmc, s)$ and $\rmL(f^\rmc \otimes \psi, s)$, respectively. We deduce that they do not vanish from Lemma~5.9 of~\cite{iwaniec-kowalski-2004} in combination with Deligne's assertion of the Ramanujan Conjecture for holomorphic elliptic modular forms~\cite{deligne-1980}.

If $l \ge k + 1$, then both~$\rmL(f^\rmc, k+l-1)$ and~$\rmL(f^\rmc \otimes \psi, l)$ in Lemma~\ref{la:petersson-scalar-product-against-eisenstein-product} do not vanish by our observations in the preceding paragraph. If~$k$ is even, we can choose the trivial Dirichlet character for $\psi$. If~$k$ is odd, we can choose an odd Dirichlet character of modulus $16$ and regard~$f$ as a modular form of level~$16 N$. As a result, we obtain the desired element $g_\infty \in ( \cE_k(N') \otimes \cE_l(N') \otimes \chi )^{\Ga_0(N')}$ with~$N' = N$ if~$k$ is even and $N' = 16 N$ if~$k$ is odd. This finishes the proof if $l \ge k + 1$.

It remains to treat the setting of~$k = l$, in which case~$\rmL(f^\rmc \otimes \psi, l)$, and if $k = l = 1$ also $\rmL(f^\rmc, k+l-1)$, in Lemma~\ref{la:petersson-scalar-product-against-eisenstein-product} are central $\rmL$-values. By Waldspurger and Kohnen-Zagier~\cite{waldspurger-1981,kohnen-zagier-1984}, there is a nonzero modular form of weight $k + 1 \slash 2$ whose $|D|$-th Fourier coefficient, for a fundamental discriminant~$D < 0$, is a nonzero multiple of~$\rmL(f^\rmc \otimes \epsilon_D, l)^{1 \slash 2}$, where $\epsilon_D$ is the Kronecker character associated with~$D$. The Sturm bound for half-integral weight modular forms implies that $\rmL(f^\rmc \otimes \epsilon_D, l)^{1 \slash 2} \ne 0$ for some $D < B(k + \frac{1}{2},N)$.

Consider the case of $k = l \ge 2$ and fix some~$D$ as in the previous paragraph. As~$k + l - 1 \ge 1 + (k+l) \slash 2$, we have verified that both~$\rmL(f^\rmc, k+l-1)$ and~$\rmL(f^\rmc \otimes \epsilon_D\psi, k+l-1)$ do not vanish. Our choice of~$D$ guarantees that~$\rmL(f^\rmc \otimes \epsilon_D, l) \ne 0$ and therefore~$\rmL(f^\rmc \otimes \epsilon_D \psi^2, l) \ne 0$. We can apply Lemma~\ref{la:petersson-scalar-product-against-eisenstein-product} with $f \leadsto f \otimes \epsilon_D\psi$ for a suitable Dirichlet character~$\psi$ mod~$16$. This yields an element~$g_{D\,\infty}$ of~$\cE_k((16D)^2 N) \otimes \cE_l((16D)^2 N)$ such that $\langle g_{D\,\infty},\, f \otimes \epsilon_D \psi \rangle \ne 0$.
By the following computation, we can choose
\begin{gather*}
  g_\infty
=
  \Big( \pi_{\mathrm{adj}}\big( \rmT_{(16D)^2}\, \iota_{\mathrm{twist}}\big( \Ind(g_{D\,\infty}) \big) \big) \Big)_\infty
\in
  \cE_k((16D)^4 N) \otimes \cE_l((16D)^4 N)
\tx{.}
\end{gather*}
Indeed, using the maps~$\pi_{\mathrm{twist}}$ and~$\iota_{\mathrm{twist}}$ from Proposition~2.19 of~\cite{raum-2017} in conjunction with~$\pi_{\mathrm{adj}}$ from Proposition~2.10 of~\cite{raum-2017}, we find the relation
\begin{gather}
\label{eq:thm:cusp-forms-in-eisenstein-product:twist-adjugate-relation}
\begin{aligned}
&
  \big\langle g_{D\,\infty},\, f \otimes \epsilon_D\psi \big\rangle
=
  \big\langle \Ind(g_{D\,\infty}),\, \Ind(f \otimes \epsilon_D\psi) \big\rangle
\\
={}&
  \Big\langle \Ind(g_{D\,\infty}),\, \pi_{\mathrm{twist}}\, \big(\rmT_{(16D)^2}\, \Ind(f) \big) \Big\rangle
\\
={}&
  \Big\langle \iota_{\mathrm{twist}}\big( \Ind(g_{D\,\infty}) \big),\, \rmT_{(16D)^2}\, \Ind(f) \Big\rangle
\\
={}&
  \Big\langle \pi_{\mathrm{adj}}\big( \rmT_{(16D)^2}\, \iota_{\mathrm{twist}}\big( \Ind(g_{D\,\infty}) \big) \big),\, \Ind(f) \Big\rangle
\\
={}&
  \Big\langle \Big( \pi_{\mathrm{adj}}\big( \rmT_{(16D)^2}\, \iota_{\mathrm{twist}}\big( \Ind(g_{D\,\infty}) \big) \big) \Big)_\infty,\, f \Big\rangle
\tx{.}
\end{aligned}
\end{gather}

We are left with the case~$k = l = 1$. By Corollary~3 of Ono-Skinner~\cite{ono-skinner-1998}, strengthening results of Waldspurger and Kohnen-Zagier, there are infinitely many fundamental discriminants~$D < 0$ co-prime to the level~$N$ of~$f$ such that $L(f^\rmc \otimes \epsilon_D, 1) \ne 0$. We fix one such~$D$. Inspecting its Fourier expansion, we find that $f \otimes \epsilon_D$ coincides with a newform of level at most $N D^2$. This allows us to apply Ono-Skinner~\cite{ono-skinner-1998} a second time, and find another fundamental discriminant~$D' < 0$ such that $L(f^\rmc \otimes \epsilon_D \otimes \epsilon_{D'}, 1) \ne 0$. Since~$\epsilon_{D'}(-1) = -1 = (-1)^k$, we can set~$\psi = \epsilon_{D'}$ and invoke Lemma~\ref{la:petersson-scalar-product-against-eisenstein-product} with $f \leadsto f \otimes \epsilon_D\psi$. The calculation in~\eqref{eq:thm:cusp-forms-in-eisenstein-product:twist-adjugate-relation} extends, showing that there is~$g_\infty \in \cE_k((DD')^2 N) \otimes \cE_l((D D')^2 N)$ such that $\langle g_\infty, f \rangle \ne 0$. This concludes the proof.
\end{proof}

\section{Eisenstein series}

In preparation for the proof of this section's main theorem, we determine the space of cusp expansions of Eisenstein series.
\begin{lemma}
\label{la:eisenstein-lower-bound}
Fix an irreducible congruence type~$\rho$ with~$T$-fixed vectors and an integer~$k$ such that $\rho = \rho^{\parity(k)}$. Then we have the inclusions of\/ $\SL{2}(\ZZ)$-representations $\rho \hra \cE_k(\rho)$ if~$k \ge 3$. If $\rho \ne \bbone$ and~$k = 2$, we also have $\rho \hra \cE_2(\rho)$. Moreover, $\cE_1(N)$ is nonzero if $N \ge 3$.
\end{lemma}
\begin{proof}
The condition~$\rho = \rho^{\parity(k)}$ ensures that the center of~$\SL{2}(\ZZ)$ acts trivially via the slash action~$|_{k,\rho}$. Consider the case~$k > 2$.  Since~$\rho$ is unitarizable, the associated vector-valued Eisenstein series of weight~$k > 2$ converge absolutely and locally uniformly. We obtain a nonzero map
\begin{gather}
\label{eq:la:eisenstein-lower-bound}
  \rho^T \otimes \rho^\vee
\lthra
  \cE_k(\rho)
\tx{,}\quad
  v \otimes w
\lmto
  w \circ E_{k,\rho,v}
\tx{.}
\end{gather}

If $k = 2$, the usual procedure of analytic continuation yields a real-analytic Eisenstein series~$E_{2,\rho,v}$ whose image under the $\xi_2$-operator of~\cite{bruinier-funke-2004} is modular of arithmetic type~$\ov{\rho}$. In particular, if~$\rho \ne \bbone$, then~$\xi_2$ annihilates~$E_{2,\rho,v}$. In other words, $E_{2,\rho,v}$ is holomorphic, and we obtain a nonzero map as in~\eqref{eq:la:eisenstein-lower-bound}.

In the case of~$k = 1$, the nonvanishing of~$\cE_1(N)$ follows from a standard computation following Section~6 of~\cite{miyake-1989}.
\end{proof}

Complementing Theorem~\ref{thm:cusp-forms-in-eisenstein-product}, which deals with cusp forms that are expressed as products of Eisenstein series, the next theorem is concerned with Eisenstein series. We achieve complete results except if~$k = l = 1$ (see Remark~\ref{rm:thm:eisenstein-series-in-eisenstein-product}).
\begin{theorem}
\label{thm:eisenstein-series-in-eisenstein-product}
Fix intergers~$k \ge 2$, $l \ge 1$, and a congruence type~$\rho$ of level~$N$. Let $N_0$ be as in Theorem~\ref{thm:cusp-forms-in-eisenstein-product}, then we have
\begin{gather}
\label{eq:thm:eisenstein-series-in-eisenstein-product}
  \rmE_{k+l}(\rho)
\;\subseteq\;
  \big(
  \cE_k(\lcm(N_0, N N_1)) \otimes \cE_l(\lcm(N_0, N_1)) \otimes \rho
  \big)^{\SL{2}(\ZZ)}
\tx{.}
\end{gather}
for $N_1$ co-prime to~$N$ chosen such that
\begin{alignat*}{2}
  N_1
&\;\ge\;
  1
\tx{,}\quad
&&
  \tx{if $k > 2$ and $l$ even;}
\\
  N_1
&\;\ge\;
  2
\tx{,}\quad
&&
  \tx{if $k > 2$ and $l > 1$ odd;}
\\
  N_1
&\;\ge\;
  2
\tx{,}\quad
&&
  \tx{if $k = 2$ and $l > 1$;}
\\
  N_1
&\;\ge\;
  3
\tx{,}\quad
&&
  \tx{if $l = 1$;}
\end{alignat*}
\end{theorem}
\begin{remark}
\label{rm:thm:eisenstein-series-in-eisenstein-product}
By symmetry, an analogous statement for $k = 1$ and $l \ge 2$ can be deduced from Theorem~\ref{thm:eisenstein-series-in-eisenstein-product}. The case $k = l = 1$, however, is not included. Computer-based experiments for small~$N$ suggest that Theorem~\ref{thm:eisenstein-series-in-eisenstein-product} should also hold in this case, if $N_1$ is sufficiently large.
\end{remark}
\begin{proof}
Given a space~$\cF$ of possibly vector-valued functions that are representable as Puiseux series, e.g., $\cF = \cE_k(N)$ or~$\cF = E_k(\rho)$, denote by~$c(\cF,0)$ the space of its constant coefficients. We will show that
\begin{gather}
\label{eq:thm:eisenstein-series-in-eisenstein-product:constant-term}
  c\big( \rmE_{k+l}(\rho),\, 0 \big)
\;=\;
  c\Big(
  \big(
  \cE_k(\lcm(N_0, N N_1)) \otimes \cE_l(\lcm(N_0, N_1)) \otimes \rho
  \big)^{\SL{2}(\ZZ)},\,
  0
  \Big)
\;\subseteq\;
  V(\rho)
\tx{.}
\end{gather}
Suppose that~\eqref{eq:thm:eisenstein-series-in-eisenstein-product:constant-term} is true, and let $f \in \rmE_{k+l}(\rho)$. Then there exists an element
\begin{gather*}
  g
\;\in\;
  \big(
  \cE_k(\lcm(N_0, N N_1)) \otimes \cE_l(\lcm(N_0, N_1)) \otimes \rho
  \big)^{\SL{2}(\ZZ)}
\end{gather*}
such that the constant term of $f - g$ vanishes. In other words, the difference~$f - g$ is a cusp form. Therefore, by Theorem~\ref{thm:cusp-forms-in-eisenstein-product} and our choice of~$N_0$, we can conclude that~$f$ is contained in the right hand side of~\eqref{eq:thm:eisenstein-series-in-eisenstein-product}. Thus we finish the proof once we have established~\eqref{eq:thm:eisenstein-series-in-eisenstein-product:constant-term}.

Observe that we have
\begin{gather*}
  \rmM_{k+l}(\rho)
\supseteq
  \big(
  \cE_k(\lcm(N_0, N N_1)) \otimes \cE_l(\lcm(N_0, N_1)) \otimes \rho
  \big)^{\SL{2}(\ZZ)}
\tx{,}
\end{gather*}
so that it follows in a straightforward way that the right hand side of~\eqref{eq:thm:eisenstein-series-in-eisenstein-product:constant-term} is contained in the left hand side. It remains to show that its left hand side is contained in the right hand side.

Equality~\eqref{eq:thm:eisenstein-series-in-eisenstein-product:constant-term} follows if it holds for all irreducibe~$\rho$. If $\rho$ has no~$T$-fixed vectors, the space of Eisenstein series~$\rmE_{k+l}(\rho)$ is zero by definition, and for this reason~\eqref{eq:thm:eisenstein-series-in-eisenstein-product:constant-term} holds. If~$\rho$ is irreducible and has~$T$-fixed vectors, then it embeds into~$\rho^\times_N$ by Lemma~\ref{la:T-fixed-vectors-yield-embedding}. We conclude that for the remainder of the proof, we may assume that~$\rho = \rho^\times_N$.

Given positive integers~$N' \isdiv N''$, we have $\cE_k(N') \subseteq \cE_k(N'')$ and $\cE_l(N') \subseteq \cE_l(N'')$. To establish~\eqref{eq:thm:eisenstein-series-in-eisenstein-product:constant-term}, it therefore suffices to show that
\begin{gather}
\label{eq:thm:eisenstein-series-in-eisenstein-product:constant-term-refined}
  c\big( \rmE_{k+l}(\rho^\times_N),\, 0 \big)
\subseteq
  c\Big(
  \big(
  \cE_k(N N_1) \otimes \cE_l(N_1) \otimes \rho^\times_N
  \big)^{\SL{2}(\ZZ)},\,
  0
  \Big)
\tx{.}
\end{gather}
Since Eisenstein series of weight~$k + l > 2$ converge absolutely, we can obtain every~$T$-fixed vector as a constant term of an Eisenstein series. This implies that
\begin{gather}
\label{eq:thm:eisenstein-series-in-eisenstein-product:constant-term-eisenstein-series}
  c\big( E_{k+l}(\rho^\times_N), 0 \big)
=
  \rho^{\times\,T\parity(k+l)}_N
\tx{.}
\end{gather}
In the case of weight~$2$, we can employ the same argument as in the proof of Lemma~\ref{la:eisenstein-lower-bound} to find that 
\begin{gather}
\label{eq:thm:eisenstein-series-in-eisenstein-product:constant-term-eisenstein-series-wt2}
  c\big( E_2(\rho^\times_N), 0 \big)
=
  \big( \rho^\times_N \ominus \bbone \big)^{T+}
\tx{.}
\end{gather}

Since~$N$ and~$N_1$ are co-prime by assumption, we can decompose $\rho^\times_{N N_1}$ as the tensor product~$\rho^\times_N \otimes \rho^\times_{N_1}$. Decomposing further by the action of the center of~$\SL{2}(\ZZ)$, we obtain the embedding
\begin{gather}
\label{eq:thm:eisenstein-series-in-eisenstein-product:tensor-decomposition}
  \rho^{\times\,\parity(k+l)}_N \otimes \rho^{\times\,\parity(l)}_{N_1}
\lhra
  \rho^{\times\,\parity(k)}_{N N_1}
\tx{.}
\end{gather}
In order to accommodate the case of~$k = 2$ and even~$l$, we refine~\eqref{eq:thm:eisenstein-series-in-eisenstein-product:tensor-decomposition} as
\begin{gather}
\label{eq:thm:eisenstein-series-in-eisenstein-product:tensor-decomposition-k2}
  \rho^{\times\,\parity(k+l)}_N \otimes \big( \rho^{\times\,\parity(l)}_{N_1} \ominus \bbone \big)
\lhra
  \rho^{\times\,\parity(k)}_{N N_1} \ominus \bbone
\tx{.}
\end{gather}

Any irreducible constituent of the left hand side of~\eqref{eq:thm:eisenstein-series-in-eisenstein-product:tensor-decomposition} and~\eqref{eq:thm:eisenstein-series-in-eisenstein-product:tensor-decomposition-k2} is a tensor product of irreducible constituents of the tensor factors, since~$N$ and~$N_1$ are co-prime. Vice versa, we infer from~\eqref{eq:thm:eisenstein-series-in-eisenstein-product:constant-term-eisenstein-series} and~\eqref{eq:thm:eisenstein-series-in-eisenstein-product:constant-term-eisenstein-series-wt2} that the tensor product of irreducible constituents of the left hand side of~\eqref{eq:thm:eisenstein-series-in-eisenstein-product:tensor-decomposition} and~\eqref{eq:thm:eisenstein-series-in-eisenstein-product:tensor-decomposition-k2} embeds into $\cE_k(N N_1)$ by Lemma~\ref{la:eisenstein-lower-bound}.

Our assumption on~$N_1$ guarantees by Lemma~\ref{la:eisenstein-lower-bound} that there is an irreducible $\SL{2}(\ZZ)$-representation $\sigma$ that embeds into $\cE_l(N_1)$. If~$k = 2$ and $l$ is even, then~$N_1$ is constrained in such a way that we can and will assume that $\sigma$ is not the trivial representation. In particular, $\rho' \otimes \sigma$ embeds into the left hand side of~\eqref{eq:thm:eisenstein-series-in-eisenstein-product:tensor-decomposition-k2} for any irreducible constituent~$\rho'$ of~$\rho^{\times\,\parity(k+l)}_N$.

Fix a $T$-fixed vector~$w_1 \in V(\sigma)$ and complete it to an orthogonal basis~$w_j$, $1 \le j \le \dim(\sigma)$ of~$V(\sigma)$. In the case of $l \ge 2$, the Eisenstein series $\td{E}_{l,\sigma^\vee,w_1^\vee} := E_{l,\sigma^\vee,w_1^\vee}$ has constant term~$w_1^\vee$. If $l = 1$, we can choose~$w_1$ in such a way that~$w_1^\vee$ is the constant term of an Eisenstein series $\td{E}_{l,\sigma^\vee,w_1^\vee}$ (in general, $\td{E}_{l,\sigma^\vee,w_1^\vee} \ne E_{l,\sigma^\vee,w_1^\vee}$ for~$l = 1$). This provides an embedding of~$\sigma$ into $\cE_l(\sigma^\vee) \subseteq \cE_l(N_1)$ via $\iota_\sigma :\, w \mto w \circ \td{E}_{l,\sigma^\vee,w_1^\vee}$. In addition, we obtain the embedding
\begin{gather}
\label{eq:thm:eisenstein-series-in-eisenstein-product:identity-inclusion}
  \bbone
\;\lhra\;
  \sigma^\vee
  \otimes
  \cE_l(N_1)
\tx{,}\quad
  1
\lmto
  \sum_{j = 1}^{\dim(\sigma)} w_j^\vee \otimes \big( w_j \circ \td{E}_{l,\sigma^\vee,w_1^\vee} \big)
\tx{.}
\end{gather}

Fix an irreducible, arbitrary constituent
\begin{gather*}
  \rho'
\;\lhra\;
  \rho^{\times\,\parity(k+l)}_N
\quad\tx{and}\quad
  v_1
\;\in\;
  V(\rho') \cap c(\rmE_{k+l}(\rho^\times_N);\,0)
\tx{.}
\end{gather*}
Observe that~$v_1 \otimes w_1$ is a $T$-fixed vector in~$V(\rho' \otimes \sigma) \subseteq V(\rho^\times_{N N_1})$. Since $k > 2$ or $\sigma \not\cong \bbone$, the Eisenstein series $E_{k, \rho' \otimes \sigma, v_1 \otimes w_1}$ exists. It allows us to define the embedding
\begin{gather*}
  \iota_{\rho' \otimes \sigma} :\,
  \rho^{\prime\,\vee} \otimes \sigma^\vee
\;\lhra\;
  \cE_k( \rho' \otimes \sigma )
\subseteq
  \cE_k(N N_1)
\tx{,}\quad
  v^\vee \otimes w^\vee
\;\lmto\;
  (v^\vee \otimes w^\vee) \circ E_{k, \rho' \otimes \sigma, v_1 \otimes w_1}
\tx{.}
\end{gather*}

Combining all the above maps we obtain the following embedding of $\SL{2}(\ZZ)$-representations:
\begin{multline}
\label{eq:thm:eisenstein-series-in-eisenstein-product:inclusion-of-invariants}
  \big(
  \rho^{\prime\,\vee}
  \otimes
  \rho^\times_N
  \big)^{\SL{2}(\ZZ)}
\;\lhra\;
  \Big(
  \rho^{\prime\,\vee}
  \otimes
  \big(
  \sigma^\vee
  \otimes
  \cE_l(N_1)
  \big)
  \otimes
  \rho^\times_N
  \Big)^{\SL{2}(\ZZ)}
\\
\;\cong\;
  \Big(
  \big(
  \rho^{\prime\,\vee}
  \otimes
  \sigma^\vee
  \big)
  \otimes
  \cE_l(N_1)
  \otimes
  \rho^\times_N
  \Big)^{\SL{2}(\ZZ)}
\;\lhra\;
  \Big(
  \cE_k(N N_1)
  \otimes
  \cE_l(N_1)
  \otimes
  \rho^\times_N
  \Big)^{\SL{2}(\ZZ)}
\tx{.}
\end{multline}

Complete~$v_1$ to an orthonormal basis~$v_i$, $1 \le i \le \dim(\rho')$ of~$V(\rho')$. Evaluating the composition of~\ref{eq:thm:eisenstein-series-in-eisenstein-product:inclusion-of-invariants}, we obtain
\begin{multline*}
  \sum_{i = 1}^{\dim(\rho')} v_i^\vee \otimes v_i
\;\lmto\;
  \sum_{i = 1}^{\dim(\rho')}
  \sum_{j = 1}^{\dim(\sigma)}
  v_i^\vee
  \otimes
  w_j^\vee \otimes \big( w_j \circ \td{E}_{l,\sigma^\vee,w_1^\vee} \big)
  \otimes
  v_i
\\
\;\lmto\;
  \sum_{i = 1}^{\dim(\rho')}
  \sum_{j = 1}^{\dim(\sigma)}
  \big( (v_i^\vee \otimes w_j^\vee) \circ E_{k, \rho' \otimes \sigma, v_1 \otimes w_1} \big)
  \otimes
  \big( w_j \circ \td{E}_{l,\sigma^\vee,w_1^\vee} \big)
  \otimes
  v_i
\tx{.}
\end{multline*}

Recall that $\iota_\sigma :\, w \mto w \circ \td{E}_{l,\sigma^\vee,w_1^\vee}$. In order to determine the constant term of the image, observe that the constant term of~$\iota_\sigma(w_j)$ equals~$1$, if $j = 1$, and $0$, otherwise, since the~$w_j$ are mutually orthogonal. Similarly, the constant term of~$\iota_{\rho' \otimes \sigma}(v_i^\vee w_j^\vee)$ equals~$1$ if $i = j = 1$, and $0$, otherwise. As a result, we directly see that the constant term of the right hand side equals~$v_1$. Since~$v_1$ and the embedding of~$\rho'$ were arbitrary, this confirms~\eqref{eq:thm:eisenstein-series-in-eisenstein-product:constant-term-refined} and finishes the proof.
\end{proof}

\renewbibmacro{in:}{}
\renewcommand{\bibfont}{\normalfont\small\raggedright}
\renewcommand{\baselinestretch}{.8}

\Needspace*{4em}
\begin{multicols}{2}
\printbibliography[heading=none]%
\end{multicols}

\Needspace*{3em}
\noindent
\rule{\textwidth}{0.15em}

{\noindent\small
Chalmers tekniska högskola och G\"oteborgs Universitet,
Institutionen för Matematiska vetenskaper,
SE-412 96 Göteborg, Sweden\\
E-mail: \url{martin@raum-brothers.eu}\\
Homepage: \url{http://raum-brothers.eu/martin}
}\\[1.5ex]
{\noindent\small
Chalmers tekniska högskola och G\"oteborgs Universitet,
Institutionen för Matematiska vetenskaper,
SE-412 96 Göteborg, Sweden\\
E-mail: \url{jiacheng@chalmers.se}
}

\end{document}

